\documentclass{amsart}

\usepackage[mathscr]{eucal}
\usepackage{amssymb}
\usepackage[usenames,dvipsnames]{color}
\usepackage{amsthm}
\usepackage{bbold}
\usepackage{enumerate}
\usepackage{array}
\usepackage{amsmath}

\usepackage[colorlinks=true,linkcolor={Brown},citecolor={Brown},urlcolor={Brown}]{hyperref}

\usepackage{cleveref}

\numberwithin{equation}{section}
\usepackage[all]{xy}
\SelectTips{cm}{}

\newdir{ >}{{}*!/-10pt/\dir{>}}

\swapnumbers 

\newtheorem{Thm}[equation]{Theorem}
\newtheorem*{Thm*}{Theorem}
\newtheorem{Prop}[equation]{Proposition}
\newtheorem{Lem}[equation]{Lemma}
\newtheorem{Cor}[equation]{Corollary}
\theoremstyle{remark}
\newtheorem{Def}[equation]{Definition}
\newtheorem{Exa}[equation]{Example}
\newtheorem{Exas}[equation]{Examples}
\newtheorem{Cons}[equation]{Construction}
\newtheorem{Hyp}[equation]{Hypothesis}

\newtheorem{Rem}[equation]{Remark}
\newtheorem{Rec}[equation]{Recollection}

\newtheorem*{Ack}{Acknowledgements}

\newcommand{\nc}{\newcommand}
\nc{\dmo}{\DeclareMathOperator}

\dmo{\coker}{coker}
\dmo{\colim}{colim}
\dmo{\cone}{cone}
\dmo{\End}{End}
\dmo{\Hom}{Hom}
\dmo{\id}{id}
\dmo{\Id}{Id}
\dmo{\Ind}{Ind}
\dmo{\Ker}{Ker}
\dmo{\opname}{op}
\dmo{\supp}{supp}
\dmo{\Spc}{Spc}
\dmo{\Spec}{Spec}

\nc{\bbZ}{\mathbb{Z}}
\nc{\cat}[1]{\mathscr{#1}}
\nc{\EndHere}{\bibliographystyle{alpha}\bibliography{TG-articles}\end{document}}
\nc{\Endcat}[1]{\End_{\cat #1}}
\nc{\Homcat}[1]{\Hom_{\cat #1}}
\nc{\ie}{{\sl i.e.}\ }
\nc{\inv}{^{-1}}
\nc{\isoto}{\overset{\sim}{\,\to\,}}
\nc{\isotoo}{\overset{\sim}{\,\too\,}}
\nc{\op}{^{\opname}}
\nc{\xto}[1]{\xrightarrow{#1}}
\nc{\oto}[1]{\overset{#1}\to}
\nc{\otoo}[1]{\overset{#1}{\,\too\,}}
\nc{\Paul}[1]{{\color{ForestGreen}#1}}
\nc{\SET}[2]{\big\{\,#1\,\big|\,#2\,\big\}}
\nc{\too}{\mathop{\longrightarrow}\limits}
\nc{\unit}{\mathbb{1}}

\dmo{\Ab}{Ab}
\dmo{\Add}{Add}
\dmo{\coev}{coev}
\dmo{\Der}{D}
\dmo{\grmod}{grmod}
\dmo{\incl}{incl}
\dmo{\im}{im}
\dmo{\Loc}{Loc}
\dmo{\modname}{mod}%
\dmo{\Mod}{Mod}
\dmo{\pr}{pr}
\dmo{\rmH}{H}
\dmo{\stab}{stab}
\dmo{\SH}{SH}
\dmo{\Stab}{Stab}
\dmo{\hname}{h}

\nc{\ababs}{{\sl ab absurdo}}
\nc{\adj}{\dashv}
\nc{\adjto}{\rightleftarrows}
\nc{\aka}{{a.\,k.\,a.}\ }
\nc{\ala}{{\sl \`a la}\ }
\nc{\bbC}{\mathbb{C}}
\nc{\bbP}{\mathbb{P}}
\nc{\cf}{{\sl cf.}\ }
\nc{\eg}{{\sl e.g.}}
\nc{\gm}{\mathfrak{m}}
\nc{\hook}{\hookrightarrow}
\nc{\ideal}[1]{\langle #1\rangle}
\nc{\ihom}{{\mathsf{hom}}} 
\nc{\ihomcat}[1]{\ihom_{\cat #1}}
\nc{\into}{\mathop{\rightarrowtail}}
\nc{\loccit}{{\sl loc.\ cit.}}
\nc{\mmod}{\text{-\!}\modname}%
\nc{\ggrmod}{\text{-\!}\grmod}%
\nc{\modd}{\modname\text{\!-}}%
\nc{\Modd}{\Mod\text{\!-}}
\nc{\MMod}{\text{-\!}\Mod}
\nc{\onto}{\mathop{\twoheadrightarrow}}
\nc{\oursetminus}{\!\smallsetminus\!}
\nc{\potimes}[1]{^{\otimes #1}}
\nc{\qquadtext}[1]{\qquad\textrm{#1}\qquad}
\nc{\quadtext}[1]{\quad\textrm{#1}\quad}
\nc{\restr}[1]{_{|_{\scriptstyle #1}}}
\nc{\sbull}{{\scriptscriptstyle\bullet}}
\nc{\sstab}{\,\text{--}\stab}%
\nc{\SpcK}{\Spc(\cat K)}
\nc{\SpcTc}{\Spc(\cat T^c)}
\nc{\then}{$\Rightarrow$}

\nc{\bbe}{\mathbb{e}}
\nc{\bbf}{\mathbb{f}}

\nc{\ohook}[1]{\,\overset{#1}\hook\,}
\nc{\oonto}[1]{\,\overset{#1}\onto\,}

\nc{\MS}{\Modd\cat S^c}
\nc{\mS}{\modd\cat S^c}
\nc{\CS}{\cat C(\cat S^c)}
\nc{\MT}{\Modd\cat T^c}
\nc{\mT}{\modd\cat T^c}
\nc{\ML}{\MMod\cat{L}}
\nc{\mL}{\mmod\cat{L}}
\nc{\bL}{\cat B_0(\cat L)}
\nc{\BL}{\cat B(\cat L)}
\nc{\cL}{\cat C_0(\cat L)}
\nc{\CL}{\cat C(\cat L)}
\nc{\cA}{\cat{A}}
\nc{\cB}{\cat{B}}
\nc{\cC}{\cat{C}}
\nc{\cD}{\cat{D}}
\nc{\cJ}{\cat{J}}
\nc{\cP}{\cat{P}}
\nc{\cS}{\cat{S}}
\nc{\cT}{\cat{T}}
\nc{\cU}{\cat{U}}
\nc{\EB}{E_{\cB}}
\nc{\EBi}{E_{\cB_i}}
\nc{\EBp}{E_{\cB'}}
\nc{\EC}{E_{\cC}}
\nc{\AB}{\bar{\cA}_{\cB}}
\nc{\ABp}{\bar{\cA}_{\cB'}}
\nc{\AC}{\bar{\cA}_{\cC}}
\nc{\hB}{\boneda_{\cB}}
\nc{\hBp}{\boneda_{\cB'}}
\nc{\hC}{\boneda_{\cC}}
\nc{\hEB}{\hat E_{\cB}}
\nc{\hEBp}{\hat E_{\cB'}}
\nc{\hEC}{\hat E_{\cC}}
\nc{\bEB}{\bar E_{\cB}}
\nc{\bEBp}{\bar E_{\cB'}}
\nc{\bEC}{\bar E_{\cC}}

\nc{\fp}{^{\textrm{fp}}}
\nc{\tensid}{$\otimes$-ideal}
\nc{\tensids}{$\otimes$-ideals}
\nc{\tensnil}{$\otimes$-nilpotent}
\nc{\yoneda}{\mathbb{h}}
\nc{\boneda}{\bar\yoneda}
\nc{\bat}[1]{\bar{\cat{#1}}}
\nc{\Loct}{\Loc^{\otimes}}
\nc{\ihombat}[1]{\ihom_{\bat #1}}
\nc{\Hombat}[1]{\Hom_{\bat #1}}
\nc{\bunit}{\bar\unit}
\nc{\hunit}{\hat\unit}
\nc{\vcorrect}[1]{{\vphantom{\vbox to #1em{}}}}

\nc{\Spch}{\Spc^{\smallh}}
\nc{\SpchT}{\Spch(\cat{T}^c)}
\nc{\SpchS}{\Spch(\cat{S}^c)}
\nc{\SpcT}{\Spc(\cat{T}^c)}
\nc{\SpcS}{\Spc(\cat{S}^c)}
\nc{\Supph}{\Supp^{\smallh}}
\nc{\SuppH}{\Supp^{\bigH}}
\nc{\SuppBF}{\Supp_{{\rm BF}}}
\nc{\Snil}{\Supp^{\smallnil}}
\nc{\tristars}{\begin{center}*\ *\ *\end{center}}

\dmo{\Img}{Im}
\dmo{\maxx}{max}
\dmo{\smallnil}{nil}
\dmo{\smallh}{h}
\dmo{\bigH}{H}
\dmo{\Supp}{Supp}
\dmo{\Serre}{Serre}

\begin{document}


\title[Homological support in big tt-categories]{Homological support of big objects\\in tensor-triangulated categories}
\author{Paul Balmer}
\date{2020 January 1}

\address{Paul Balmer, Mathematics Department, UCLA, Los Angeles, CA 90095-1555, USA}
\email{balmer@math.ucla.edu}
\urladdr{http://www.math.ucla.edu/$\sim$balmer}

\begin{abstract}
Using homological residue fields, we define supports for big objects in tensor-triangulated categories and prove a tensor-product formula.
\end{abstract}

\subjclass[2010]{18D99; 20J05, 55U35}
\keywords{Tensor-triangular geometry, homological residue field, big support}

\thanks{Research supported by NSF grant~DMS-1901696.}

\maketitle


\vskip-\baselineskip\vskip-\baselineskip
\tableofcontents


\vskip-\baselineskip\vskip-\baselineskip
\section{Introduction}


%
\begin{Hyp}
\label{Hyp:big-tt}%
Let $\cT$ be a \emph{`big' tensor-triangulated category}, meaning a rigidly-compactly generated one, as in~\cite{BalmerFavi11}. So~$\cT$ admits small coproducts; its subcategory~$\cT^c$ of compact objects coincides with that of rigid (strongly-dualizable) objects; $\cT^c$ is essentially small and generates~$\cT$ as a localizing subcategory.
\end{Hyp}

Here are our main results. Explanations are given after the statement.
\begin{Thm}[{\Cref{sec:support}}]
\label{Thm:support-intro}%
One can assign to every object $X$ of~$\cT$ a subset~$\Supp(X)$ of the homological spectrum $\SpchT$ of~\cite{Balmer20a}, with the following properties:
\begin{enumerate}[\rm(a)]
\smallbreak
\item
\label{it:support-agree}%
For every compact $x\in \cT^c$, this support agrees with the usual one in~$\SpcT$, as in~\cite{Balmer05a}. In particular $\Supp(0)=\varnothing$ and $\Supp(\unit)=\SpchT$.
\medbreak
\item
\label{it:support-coprod}%
For every family $\{X_i\}_{i\in I}$ in~$\cT$, we have $\Supp(\,\bigsqcup_{i\in I}X_i)=\bigcup_{i\in I}\Supp(X_i)$.
\medbreak
\item
\label{it:support-triangle}%
For every exact triangle $X\to Y\to Z\to \Sigma X$ in~$\cT$, we have $\Supp(Z)\subseteq\Supp(X)\cup \Supp(Y)$. Moreover $\Supp(\Sigma X)=\Supp(X)$.
\medbreak
\item
\label{it:support-TPF}%
The Tensor-Product Formula holds: for every $X,Y\in\cT$ we have
\begin{equation*}
\Supp(X\otimes Y)=\Supp(X)\cap \Supp(Y).
\end{equation*}
\end{enumerate}
\end{Thm}

\goodbreak

In order to appreciate the homological spectrum~$\SpchT$, in which our support theory takes its values, let us give some context.

Big tt-categories~$\cT$ are used across homotopy theory, algebraic geometry and representation theory. They appear as `unital algebraic stable homotopy categories' in~\cite{HoveyPalmieriStrickland97}. More recent examples include derived categories of motives and stable $\mathbb{A}^{\!1}$-homotopy categories. Symmetric monoidal presentable stable $\infty$-categories~\cite{LurieHA} provide another possible source of examples.

In all cases, the optimal support theory on the essentially small subcategory~$\cT^c$ of compact objects is the one borne by the triangular spectrum $\SpcT$ of~\cite{Balmer05a}. This space $\SpcT$ is now known in many examples; see the survey~\cite{Balmer19}. On the other hand, it is not clear how to properly define the support of non-compact objects in general. This is the problem we want to address here.

In the famous BIK series~\cite{BensonIyengarKrause08,BensonIyengarKrause11c,BensonIyengarKrause11a,BensonIyengarKrause12a,BensonIyengarKrause12b}, Benson, Iyengar and Krause approach the question via a \textsl{deus ex machina}: They assume the existence of a \emph{noetherian} ring~$R$ acting nicely on~$\cT$. Then BIK define a support theory with values in~$\Spec(R)$ and prove many strong results, that apply particularly well to the representation theory of finite groups over fields. However the BIK setup is somewhat restrictive. Unsurprisingly, it does not cover derived categories of non-noetherian schemes -- but who cares about non-noetherian schemes? The real drawback is that some very reasonable tt-categories~$\cT$ are not stratified by any noetherian ring~$R$. Even in representation theory of finite groups, replacing the field of scalars by a commutative ring (like~$\bbZ$) sends the deus reeling in the machina, as discovered by BIK themselves in~\cite{BensonIyengarKrause13}. More importantly, topologists have long known that the chromatic tower of the stable homotopy category~$\SH$ is not a noetherian phenomenon. And $\SH$ is the initial tt-category: What happens in~$\SH$ has repercussions throughout the field. So the general problem remains wide open and important, beyond the BIK setting.

In the joint work with Favi~\cite{BalmerFavi11} and in Stevenson~\cite{Stevenson13}, the spectrum of the BIK ring~$R$ is replaced by the more canonical~$\SpcT$. A support for big objects was proposed in~\cite{BalmerFavi11} but we could not prove the Tensor-Product Formula for it. So, among the properties listed in \Cref{Thm:support-intro}, the most remarkable is probably~\eqref{it:support-TPF}.

In recent years, new tools have emerged, like the \emph{homological residue fields} of~\cite{BalmerKrauseStevenson19,Balmer20a}. These consist of homological tensor-functors
\begin{equation}
\label{eq:hB-intro}%
\hB\colon \cT\to \bat{A}_{\cB}
\end{equation}
from our big tt-category~$\cT$ to various tensor-\emph{abelian} categories~$\AB$. The parameter~$\cB$ lives in the aforementioned \emph{homological spectrum}~$\SpchT$ and the abelian categories~$\AB$ are `simple' (\Cref{Rem:simple}), as one would expect of the category of vector spaces over a field, for instance. We review this material in \Cref{Rec:Spch}. For now, suffice it to say that these functors~$\hB$ are abstract versions of:
\begin{itemize}
\item[--] ordinary residue fields in algebraic geometry,
\item[--] Morava $K$-theories in homotopy theory,
\item[--] cyclic shifted subgroups and $\pi$-points in modular representation theory.
\end{itemize}
They also give rise to a Nilpotence Theorem~\cite[Thm.\,1.1]{Balmer20a}. In summary, the homological spectrum $\SpchT$ and the residue fields~$\hB$ have a life of their own: They were not invented for the sake of the present paper. This homological spectrum~$\SpchT$ is also very close to the triangular one. Indeed there is a map
\begin{equation}
\label{eq:phi}%
\phi\colon\SpchT\onto \SpcT
\end{equation}
that is always surjective and actually bijective in all known examples, see~\cite[\S\,5]{Balmer20a}. So in first approximation, the reader can think of $\SpchT$ as equal to the more familiar triangular spectrum $\SpcT$ of compact objects. In second approximation, \Cref{App:hom-local} gives a reformulation of injectivity of~$\phi$. (This also explains the meaning of agreement on compacts~\eqref{it:support-agree}; see details in~\Cref{Prop:agree}.)

Following the sibylline suggestion of~\cite[Remark~4.6]{Balmer20a}, it is tempting to define the support of every big object~$X$ in~$\cT$ as the following subset of~$\SpchT$
\begin{equation}
\label{eq:naive-support}%
\SET{\cB\in\SpchT}{\hB(X)\neq 0}.
\end{equation}
This `naive' support is almost the right thing to do. It will work fine for small objects and for ring objects but it might still fail the Tensor-Product Formula. Our construction ends up being one notch more involved.

To explain how $\Supp(X)$ is constructed, we need to know a little more about the homological residue fields~$\hB$ of~\eqref{eq:hB-intro} and its target category~$\AB$. In that `residue' Grothendieck category~$\AB$, the subcategory of finitely presented objects~$\AB\fp$ has only $0$ and~$\AB\fp$ as Serre \tensids\ (\Cref{Rem:simple}) but a similar property for the big category~$\AB$ is not known to be true, nor is it really expected. However, $\AB$ admits a \emph{unique maximal localizing \tensid} (\Cref{Thm:max-to-max}). Our definition of the support of an object~$X$ in~$\cT$ is the collection of those~$\cB$ in~$\SpchT$ where $X$ does not belong to that unique maximal localizing \tensid.

One can make this more explicit in terms of~$\cT$. In~$\AB$, the $\otimes$-unit~$\bunit$ admits an injective hull,~$\bEB=\hB(\EB)$, that comes via~$\hB$ from a canonical pure-injective object~$\EB$ in~$\cT$. One has $\hB(X)=0$ if and only if $X\otimes \EB=0$. So the `naive' support of~\eqref{eq:naive-support} is $\SET{\cB}{X\otimes\EB\neq0}$. Our support is defined as
\begin{equation}
\label{eq:Supp}%
\Supp(X)=\SET{\cB\in\SpchT}{[X,\EB]\neq0}
\end{equation}
where $[-,-]$ stands for the internal-hom in~$\cT$. It is a subset of the naive support.

\medbreak

When given a support theory, it is natural to wonder whether $\Supp(X)=\varnothing$ implies $X=0$. We point out that there is no hope for such a result in our glorious generality, if one wants the Tensor-Product Formula. Indeed, if $\cT$ contains a non-zero object~$X$ such that $X\otimes X=0$ then $\Supp(X)$ must be empty by~\eqref{it:support-TPF}. Neeman~\cite{Neeman00} gives examples of such~$X\neq0$ with $X\otimes X=0$ in derived categories $\cT=\Der(R)$ of commutative rings~$R$. The Brown-Comenetz dual of the sphere is another example of such an object~$X$ in~$\cT=\SH$ itself, see~\cite[\S\,7]{HoveyStrickland99}.

Things are a little nicer with ring objects, as we now explain.

\begin{Def}
\label{Def:weak-ring}%
We say that an object~$A$ in~$\cT$ with a map~$\eta\colon \unit\to A$ (its `unit') is a \emph{weak ring} if $A\otimes\eta\colon A\to A\otimes A$ is a split monomorphism (whose retraction $A\otimes A\to A$ can be thought of as a unital non-associative multiplication on~$A$).
\end{Def}

Of course, actual ring objects are weak rings. The pure-injective objects~$\EB$ discussed above are weak rings as well, although they are not known to be rings in general. We then prove in~\Cref{Thm:rings}:
\begin{Thm}
\label{Thm:rings-intro}%
For all (weak) rings~$A$, the support coincides with the naive support
\[
\Supp(A)=\SET{\cB\in\SpchT}{\hB(A)\neq 0}.
\]
Furthermore, if $\Supp(A)=\varnothing$ then $A=0$.
\end{Thm}

In other words, our support theory is particularly effective for (weak) rings. As an application, we revisit the problem of determining the image of the map of spectra induced by a tt-functor, see~\cite{Balmer18a}. Let $F\colon \cT\to \cS$ be a tensor-triangulated functor admitting a right adjoint~$U\colon \cS\to \cT$. Note that $F$ restricts to compact-rigid objects $F\colon\cT^c\to \cS^c$. As~$U$ is lax-monoidal, $U(\unit_\cS)$ is a ring object in~$\cT$. In~\cite{Balmer18a}, it is shown that when $U(\unit)$ is compact then $\supp(U(\unit))$ in~$\SpcT$ coincides with the image of the map $\SpcS\to \SpcT$ induced by~$F$. However, this assumption that the right adjoint~$U$ maps~$\unit$ to a compact object is very, very restrictive. We prove here an unconditional generalization:
\begin{Thm}[{\Cref{Thm:Spc-image}}]
\label{Thm:Spc-image-intro}%
As above, let $F\colon \cT\to \cS$ be a tt-functor between `big' tt-cate\-gories, with right-adjoint $U\colon \cS\to \cT$. Then the image of the map~$\Spch(F)\colon$ $\SpchS\to \SpchT$ is exactly the support $\Supp(U(\unit))$ of the ring object~$U(\unit)$. Consequently, the image of $\Spc(F)\colon \SpcS\to \SpcT$ is $\phi(\Supp(U(\unit)))$.
\end{Thm}

\begin{Ack}
I am very thankful to Greg Stevenson for his comments, and in particular for catching an excessively enthusiastic claim in a previous version of this work.
\end{Ack}

\goodbreak
\section{Yoneda and modules}
\label{sec:modules}%
\medbreak

Many readers can safely skip this section and refer back to it as needed, especially those familiar with the module category
\[
\cA:=\MT=\Add((\cT^c)\op,\Ab)
\]
of additive contravariant functors from~$\cT^c$ to abelian groups.

\begin{Rec}
\label{Rec:MT}%
The abelian category $\cA=\MT$ is a Grothendieck category, whose subcategory of finitely presented objects $\cA\fp=\mT$ coincides with the Freyd envelope of~$\cT^c$ (\cite[Chap.\,5]{Neeman01}). The (restricted) Yoneda embedding~$\yoneda$
\[
\xymatrix@H=1.5em@R=1em{
\cT^c \quad \ar@{^(->}[r]^-{\yoneda} \ar@{^(->}@<-.5em>[d]
& \quad \mT=\cA\fp \ar@{^(->}@<-.5em>[d]
\\
\cT \quad \ar[r]^-{\yoneda}
& \quad \MT=\cA
}
\]
is defined by $\yoneda(X)=\hat X$ where $\hat X=\Homcat{T}(-,X)\restr{\cT^c}$. This functor $\yoneda\colon \cT\to \cA$ is homological (maps distinguished triangles to exact sequences), preserves coproducts and is universal among those (\cite[Cor.\,2.4]{Krause00}). It is also conservative.

Restricted-Yoneda~$\yoneda$ is fully faithful on~$\cT^c$, and identifies the latter with finitely presented projective objects in~$\cA$. Every (big) object $M\in\cA$ is a filtered colimit of finitely presented objects. (Indeed $\cA$ is \emph{locally coherent}; see~\cite[A.7]{BalmerKrauseStevenson17app}.) Also, every object~$M\in\cA$ is a quotient of a coproduct $\sqcup_{i\in I}\,\hat x_i$ of rigid-compact objects~$x_i\in \cT^c$. For an object $P=\sqcup_{i\in I}\,x_i$ with all~$x_i\in\cT^c$ (or a summand of such a coproduct) and for $Y\in \cT$ arbitrary, restricted-Yoneda yields an isomorphism
\begin{equation}
\label{eq:Hom-proj}%
\yoneda\colon\Homcat{T}(P,Y)\isoto \Homcat{A}(\hat P,\hat Y)\,.
\end{equation}
Hence all projectives in~$\cA$ are $\hat P$ for $P$ a summand of some~$\sqcup_{i\in I}\,x_i$ with all $x_i\in\cT^c$.

The category~$\cA$ also has enough injectives and they also come from~$\cT$. By~\cite{Krause00}, they are all of the form~$\hat E$ for a unique~$E\in\cT$, called a \emph{pure-injective}. For every object~$X$ and every pure-injective~$E$ in~$\cT$, restricted-Yoneda gives an isomorphism
\begin{equation}
\label{eq:Hom-pure-inj}%
\yoneda\colon\Homcat{T}(X,E)\isoto \Homcat{A}(\hat X,\hat E)\,.
\end{equation}
\end{Rec}

\begin{Rec}
\label{Rec:tensor-on-MT}%
An essential feature of the module category is its tensor product, obtained by Day convolution, and discussed in~\cite[App.\,A]{BalmerKrauseStevenson17app}. This tensor is colimit-preserving in each variable, in particular it is right-exact. It makes the restricted-Yoneda functor (not just the part on~$\cT^c$) into a tensor functor $\yoneda\colon \cT\to \cA$ and every object in the image of~$\cT$ is $\otimes$-flat, \ie $\hat X\otimes-$ is exact for all~$X\in\cT$.

As a consequence of this, all projective objects of~$\cA$ and, perhaps more remarkably, all injective objects of~$\cA$ are $\otimes$-flat.

By general Grothendieck-category theory $\cA$ is then \emph{closed} monoidal, \ie it admits an internal-hom functor $[-,-]_\cA\colon \cA\op\times \cA\too\cA$. Beware that $\yoneda$ might not be a closed functor but one can easily upgrade~\eqref{eq:Hom-pure-inj} into an isomorphism
\begin{equation}
\label{eq:[-,E]}%
\yoneda([X,E]_\cT)\cong [\hat X,\hat E]_\cA
\end{equation}
for every $X,E\in\cT$ with $E$ pure-injective, by testing via $\Homcat{A}(\hat c,-)$ with $c\in\cT^c$ and using that $\yoneda$ is a tensor functor.
\end{Rec}

\begin{Rec}
\label{Rec:A/C}%
Let $\cC$ in~$\cA$ be a localizing subcategory (\ie closed under coproducts, extensions, subobjects and quotients). We have a Gabriel quotient~$\cA/\cC$
\begin{equation}
\label{eq:A/C}%
\vcenter{\xymatrix@R=1.8em{
\cat{A} \ar@<-.3em>@{->>}[d]_-{Q_\cC}
\\
\cA/\cat{C} {}^{\vcorrect{.6}} \ar@<-.3em>@{_(->}[u]_-{R_\cC}
}}
\end{equation}
where $Q_\cC$ is the universal exact functor with kernel~$\cC$, its right adjoint $R_\cC$ is fully-faithful and $Q_\cC\circ R_\cC\cong\Id_{\cA/\cC}$. We shall only consider subcategories $\cC$ that are $\otimes$-ideal ($\cA\otimes \cC\subseteq\cC$), in which case $\cA/\cC$ inherits a unique tensor structure such that $Q_\cC$ is a tensor functor. When $\cC$ is clear from the context, we often write $\bar{X}$ instead $Q_\cC(\hat X)$ for $X\in\cT$. We also write $\bat{A}_{\cC}$ (or just $\bat{A}$) for $\cA/\cC$ and
\begin{equation}
\label{eq:hC}%
\boneda_{\cC}\colon \cT\too \bat{A}_\cC=\cat{A}/\cC
\end{equation}
for the composite $Q_\cC\circ \yoneda$, a coproduct-preserving cohomological tensor-functor.
\end{Rec}

Those quotients $\cA/\cC$ inherit a number of properties that hold for~$\cA$.
\begin{Prop}
\label{Prop:A/C}%
With notation as in~\Cref{Rec:A/C}, we have:
\begin{enumerate}[\rm(a)]
\item
\label{it:A/C-a}%
The Gabriel quotient $\bat{A}=\cA/\cC$ is a Grothendieck category, that is closed monoidal. Its tensor product preserves colimits in each variable.
\item
\label{it:A/C-b}%
Every $X\in\cT$ has $\otimes$-flat image $\bar X=Q_\cC(\hat X)$ in~$\bat{A}$. If $x\in \cT^c$ then $\bar x$ is rigid.
\item
\label{it:A/C-c}%
Every object of~$\bat{A}$ is a quotient of a coproduct $\sqcup_{i\in I}\,\bar x_i$ with all $x_i\in \cT^c$.
\item
\label{it:A/C-d}%
Every injective object of~$\bat{A}$ is of the form~$\bar E$ for some pure-injective~$E$ in~$\cT$, uniquely characterized by the property~$\hat E\cong R_\cC(\bar E)$ in~$\cA$. In particular, $\bar E$ is $\otimes$-flat. Furthermore, $\boneda_\cC\colon \cT\to \bat{A}$ induces, for every $X\in\cT$, an isomorphism
\[
\Homcat{T}(X,E)\isoto \Homcat{{\bat{A}}}(\bar X,\bar E).
\]
\item
\label{it:A/C-e}%
The right adjoint $R_\cC\colon \bat{A}\to \cat{A}$ is lax-monoidal and closed; more precisely for every $M\in\cA$ and $N\in\bat{A}$, if we denote by $[-,-]$ the internal-homs, we have
\[
[M,R_\cC(N)]_{\cA}\cong R_{\cC}([Q_\cC(M),N]_{\bat{A}}).
\]
\item
\label{it:A/C-f}%
For every injective $\bar E$ in~$\bat{A}$, the internal-hom $[-,\bar E]_{\bat{A}}\colon \bat{A}\op\to \bat{A}$ is exact.
\end{enumerate}
\end{Prop}

\begin{proof}
All these properties follow easily from \cite[App.\,A]{BalmerKrauseStevenson17app} and \cite[\S\,2\textrm{ and }App.\,A]{BalmerKrauseStevenson19}. We need to prove~\eqref{it:A/C-f}, for~\cite[Lem.\,2.8]{BalmerKrauseStevenson19} is only stated for~$\cA$. We have
\[
[-,\bar E]_{\bat{A}}\circ Q_{\cC}=[Q_{\cC}(-),\bar E]_{\bat{A}}\cong Q_{\cC}R_{\cC}([Q_\cC(-),\bar E]_{\bat{A}})\underset{\eqref{it:A/C-e}}{\cong} Q_{\cC}\circ[-,R_\cC(\bar E)]_{\cat{A}}
\]
and the latter is exact by~\cite[Lem.\,2.8]{BalmerKrauseStevenson19}, since $R_\cC\bar E$ is injective in~$\cA$ itself. As precomposition with $Q_\cC$ detects exactness for functors on~$\bat{A}$, we get the result.
\end{proof}

\begin{Rem}
\label{Rem:Sigma}%
Note that $\Sigma\colon \cT\to \cT$ induces an auto-equivalence $\Sigma\colon \cA\to \cA$ which is isomorphic to~$\yoneda(\Sigma\unit)\otimes-$. Hence all our \tensids~$\cC$ in~$\cA$ are stable under~$\Sigma$ and all quotients inherit a suspension~$\Sigma\colon \AC\isoto \AC$ such that $Q_\cC\Sigma\cong\Sigma Q_\cC$.
\end{Rem}

An important construction is the one of~\cite[\S\,3]{BalmerKrauseStevenson19}, slightly generalized:
\begin{Cons}
\label{Cons:E_C}%
For $\cC\subseteq\cA$ localizing \tensid, the injective hull of~$\bar{\unit}$ in~$\AC$ comes from~$\hC$, by \Cref{Prop:A/C}\,\eqref{it:A/C-d}. So there exists a morphism
\[
\eta_\cC\colon \unit\to \EC
\]
in~$\cT$ such that $\bar\eta_\cC=\hC(\eta_\cC)\colon \bar{\unit}\into \bar E_\cC$ is the injective hull. It is characterized by $R_\cC(\bEC)\cong\hEC$ in~$\cA$. Since $\bEC$ is flat and injective, $\bEC\otimes\bar\eta_{\cC}$ is a split monomorphism in~$\AC$, hence by \Cref{Prop:A/C}\,\eqref{it:A/C-d} again, the retraction exists in~$\cT$, meaning that $(\EC,\eta_\cC)$ is a \emph{weak ring} in~$\cT$ in the sense of \Cref{Def:weak-ring}.
\end{Cons}

\begin{Rec}
\label{Rec:A/B}%
A particular class of localizing $\otimes$-ideals $\cC\subseteq\cA$ are those generated by Serre $\otimes$-ideals $\cB\subseteq\cA\fp=\mT$ of finitely presented objects (where \tensid\ only means $\cA\fp\otimes\cB\subseteq\cB$ of course). Explicitly $\cC=\Loc(\cB)$ is the smallest localizing subcategory containing~$\cB$, which is then automatically \tensid\ in~$\cA$. In some notation, typically in indices, we drop the `$\Loc$' part and write only~$\cB$, like for instance with the canonical cohomological tensor-functor~\eqref{eq:hC}:
\begin{equation}
\label{eq:hB}%
\boneda_{\cB}\colon \cT\too \bat{A}_\cB:=\cat{A}/\Loc(\cB)\,.
\end{equation}
Similarly, $\EB$ means $E_{\Loc(\cB)}$ as in \Cref{Cons:E_C} for~$\cC=\Loc(\cB)$.
\end{Rec}

Under these additional assumptions, we know more than in \Cref{Prop:A/C}:

\begin{Prop}
\label{Prop:A/B}%
With notation as in~\Cref{Rec:A/B}, we have:
\begin{enumerate}[\rm(a)]
\item
\label{it:A/B-a}%
The Grothendieck category $\AB$ is still locally coherent. Its finitely presented objects $\bat{A}\fp=\cA\fp/\cB$ coincide with the Gabriel quotient of~$\cA\fp$ by~$\cB$.
\item
\label{it:A/B-b}%
We can recover $\cB$ and $\Loc(\cB)$ from the pure-injective weak ring~$E_\cB$ of~\Cref{Cons:E_C}, as $\cB=\SET{M\in\cA\fp}{\hEB\otimes M=0}$ and $\Loc(\cB)=\Ker(\hEB\otimes-)$.
\item
\label{it:A/B-c}%
Every subobject of a finitely presented object in~$\AB$ is the colimit of its finitely presented subobjects. In particular, if $I\into \bar\unit$ in~$\AB$ then $I=\mathop{\colim}\limits_{\textrm{f.p.\,}M\into I}M$.
\end{enumerate}
\end{Prop}

\begin{proof}
All this is in \cite[\S\,3]{BalmerKrauseStevenson19}.
\end{proof}

We shall need the following general observation.

\begin{Lem}
\label{Lem:Ker[-,E]}%
Let $E$ be an injective object in~$\cA$. Then the subcategory $\Ker [-,E]_{\cA}=\SET{M\in\cA}{[M,E]_{\cA}=0}$ is a localizing $\otimes$-ideal of~$\cA$.
\end{Lem}

\begin{proof}
This kernel $\Ker[-,E]$ is a localizing (in particular, Serre) subcategory because the internal-hom functor $[-,E]$ turns coproducts into products and is exact by \Cref{Prop:A/C}\,\eqref{it:A/C-f} (or \cite[Lem.\,2.8]{BalmerKrauseStevenson19}). Furthermore, this kernel is closed under tensoring with every rigid object~$\hat x\in\cA$ since $[\hat x\otimes M,E]\cong \hat x^\vee\otimes[M,E]$. Now for a general object~$N$, there exists an epimorphism $\sqcup_{i\in I}\hat x_i\onto N$ with all~$\hat x_i$ rigid by \Cref{Rec:MT}. Hence by right-exactness of the tensor we get an epimorphism $\sqcup_{i\in I}\hat x_i\otimes M\onto N\otimes M$ for every $M\in\cA$. If $M$ belongs to $\Ker([-,E])$ we know that $\sqcup_{i\in I}\hat x_i\otimes M$ also does, hence so does $N\otimes M$ as the kernel is already known to be closed under quotients, since it is a Serre subcategory.
\end{proof}

\begin{Rem}
The same statement holds in any $\AC$, with a similar proof.
\end{Rem}

\goodbreak
\section{Maximal localizing tensor-ideals}
\label{sec:maximal}%
\medbreak

We recall the homological spectrum~$\SpchT$ and prove that every prime $\cB\in\SpchT$ is contained in a unique maximal localizing \tensid\ (\Cref{Thm:max-to-max}).

\begin{Rec}
\label{Rec:Spch}%
The \emph{homological spectrum}~$\SpchT$ consists of all maximal Serre \tensids\ $\cB$ of the abelian subcategory~$\cA\fp=\mT$ of finitely presented $\cT^c$-modules. We call those~$\cB$ the \emph{homological primes} of~$\cT$. Note that they only depend on the subcategory~$\cT^c$. For each $\cB\in\SpchT$, its preimage in~$\cT^c$ under Yoneda
\[
\phi(\cB)=\yoneda\inv(\cB)\cap \cT^c=\SET{x\in\cT^c}{\hat x\in\cB}
\]
is a (triangular) prime ideal in~$\cT^c$. This defines a surjection~$\phi\colon\SpchT\onto \SpcT$ by~\cite[Cor.\,3.9]{Balmer20a}. Each $\cB\in\SpchT$ yields a coproduct-preserving homological tensor-functor $\hB\colon \cT\to \AB=\cA/\Loc(\cB)$ as in~\eqref{eq:hB}.
\end{Rec}

\begin{Exas}
A list of examples of big tt-categories appears in~\cite[\S\,1.2]{HoveyPalmieriStrickland97}.

We said in the introduction that the functors $\hB$ provide an abstract form, for any big tt-category, of ordinary residue fields in algebraic geometry, of Morava $K$-theories in stable homotopy theory and of $\pi$-points in modular representation theory. In fact the $\hB$ \emph{improve} those examples in that they are \emph{always} tensor functors (\ie symmetric monoidal functors) whereas the functors induced by Morava $K$-theories are sometimes not symmetric monoidal (at the prime~2) and $\pi$-points almost never are. (Also, $\pi$-points are only well-defined up to some notion of equivalence whereas the $\hB$ are intrinsical. And the $\hB$ always give us a Nilpotence Theorem, which was not known with $\pi$-points.) See further details in~\cite[\S\,5]{Balmer20a}.

Now, it is one thing to know that $\phi$ is a bijection $\SpchT\isoto \SpcT$ and thus to know `how many' homological primes~$\cB$ there are in the many examples listed above. It is another thing to describe the functors~$\hB$ and the weak rings~$\EB$ in~$\cT$ in concrete terms. The latter project is the subject of the upcoming work~\cite{BalmerCameron20pp}.
\end{Exas}

\begin{Rem}
A puzzling feature of the examples treated in~\cite{Balmer20a} is that in all cases $\phi$ is a bijection. We do not know how general that is but we translate this property in relatively down-to-earth terms in \Cref{Thm:hom-local} in \Cref{App:hom-local}.
\end{Rem}

\begin{Rem}
\label{Rem:simple}%
The maximality of~$\cB\in\SpchT$ among the Serre \tensids\ of~$\cA\fp$ tells us that $\AB\fp$ is `simple': it has only the two trivial Serre \tensids, zero and~$\AB\fp$. (See \Cref{Rec:A/B} if necessary.) Simplicity has the following consequence.
\end{Rem}

\begin{Prop}
\label{Prop:ring-in-AB}%
Let $\cB\in\SpchT$ and $A\in\AB$ be a weak ring (\Cref{Def:weak-ring}) that is $\otimes$-flat, like $A=\bar E$ for some weak ring~$E\in\cT$. Then either $A=0$ or its unit is a monomorphism $\bunit\into A$.
\end{Prop}

\begin{proof}
This is a basic trick in the proof of the Nilpotence Theorem~\cite{Balmer20a}. Consider the exact sequence $\Ker(\eta)\into \bunit\xto{\eta} A$ in~$\AB$. Since $A$ is flat and $A\otimes\eta$ is a (split) monomorphism, we have $A\otimes\Ker(\eta)=0$. Suppose that $\Ker(\eta)\neq0$ and let us show that $A=0$. By \Cref{Prop:A/B}\,\eqref{it:A/B-c}, we know that $\Ker(\eta)$ is the colimit of its finitely presented subobjects. Take a finitely presented $M\neq0$ with $M\into \Ker(\eta)$. By flatness of~$A$ and $A\otimes\Ker(\eta)=0$, we see that $A\otimes M=0$. Hence $\Ker(A\otimes-)\cap \AB\fp$ is a non-zero Serre (by flatness) \tensid\ of~$\AB\fp$. As $\AB\fp$ is `simple' (\Cref{Rem:simple}) this forces $\Ker(A\otimes-)\cap \AB\fp=\AB\fp$ to contain~$\bunit$, giving~$A=0$.
\end{proof}

\begin{Rem}
\label{Rem:max-vs-max}%
During the year the author spent in Bielefeld working with Krause and Stevenson on~\cite{BalmerKrauseStevenson19}, we hesitated between maximal Serre \tensids\ $\cB\subset\cA\fp$ and maximal localizing \tensids\ $\cC\subset\cA$. We opted for the finitely presented ones for the extra properties of \Cref{Prop:A/B}, that turned out to be useful in~\cite{BalmerKrauseStevenson19} and later in the proof of the Nilpotence Theorem~\cite{Balmer20a}. It remains an open question to relate the two notions. Let us first see that every localizing \tensid\ bigger than a homological prime has the `same' pure-injective weak ring~$\EB$.
\end{Rem}

\begin{Prop}
\label{Prop:EB=EC}%
Let $\cB\in\SpchT$ and $\cC\subset\cA$ be a proper localizing \tensid, containing~$\cB$ or equivalently containing~$\Loc(\cB)$. Let $\EB$ and $\EC$ be the corresponding weak rings of \Cref{Cons:E_C}. Then there is an isomorphism $\EB\simeq \EC$ in~$\cT$ compatible with the unit maps~$\unit\to\EB$ and~$\unit\to \EC$.
\end{Prop}

\begin{proof}
Since $\Loc(\cB)\subseteq\cC$, we can perform the Gabriel quotient $\cA/\cC$ in two steps:
\[
\kern4em\xymatrix@R=2em{
\cat{A} \ar@<-.3em>@{->>}[d]_-{Q_{\cB}} \ar@/_2em/@<-1em>@{->>}[dd]_-{Q_{\cC}}
&&& \hEC
\\
\AB \ar@<-.3em>@{_(->}[u]_-{R_\cB} \ar@<-.3em>@{->>}[d]_-{\bar Q}
&&& \tilde E \ar@{|->}[u]
\\
\AC \ar@<-.3em>@{_(->}[u]_-{\bar R} \ar@/_2em/@<-1em>@{_(->}[uu]_-{R_{\cC}}
&&& \bEC \ar@{|->}[u]
}
\]
for an intermediate Gabriel localization~$\bar Q\adj\bar R$. Let $\tilde E:=\bar R(\bEC)$ the injective object in~$\AB$ associated to the injective hull~$\bunit\into \bEC$ in~$\AC$. By construction of~$\EC\in\cT$, we have $\hEC\cong R_{\cC}(\bEC)\cong R_{\cB}(\tilde E)$ and therefore $\tilde E\cong Q_\cB(\hEC)=\hB(\EC)$. Since $\EC$ is a weak ring in~$\cT$, its image $\tilde E$ is a weak ring in~$\AB$, that is non-zero since $\bar Q(\tilde E)\cong\bEC\neq0$. By \Cref{Prop:ring-in-AB}, the unit $\hB(\eta_{\cC})\colon\bunit\into \tilde E$ is therefore a monomorphism in~$\AB$. On the other hand, $\hB(\eta_\cB)\colon\bunit\into\bEB$ is by definition the injective hull of~$\bunit$ in~$\AB$. Hence there exists a commutative diagram in~$\AB$
\begin{equation}
\label{eq:aux-max-vs-max}%
\vcenter{\xymatrix@C=4em@R=2em{\unit \ar@{ >->}[r]^-{\hB(\eta_\cB)} \ar@{ >->}[rd]_-{\hB(\eta_{\cC})}
& \bEB \ar@{ >->}[d]^-{\exists\,\varphi}
\\
& \tilde E
}}
\end{equation}
where $\varphi$ is a split monomorphism. Applying the exact functor~$\bar Q$ to this diagram and using that~$\bar Q(\tilde E)=\bEC$ is the injective hull of~$\bunit$ in~$\AC$, we see that $\bar Q(\varphi)$ must be an isomorphism. In summary, $\tilde E\cong \bEB\oplus N$ in~$\AB$ for an object~$N$ such that $\bar Q(N)=0$. But since $\tilde E=\bar R(\bEC)$ and $\bar R$ is fully faithful, this forces the object~$N$ to be in the image of~$\bar R$ as well (it is `$\bar R$-local'). It follows that $N\cong \bar R\bar Q(N)=0$. Therefore $\varphi$ is already an isomorphism $\bEB\simeq \tilde E\simeq \bar R(\bEC)$ in~$\AB$. By \Cref{Prop:A/C}\,\eqref{it:A/C-d} (applied to $\bat{A}=\cA/\Loc(\cB)$) the isomorphism $\varphi$ in~\eqref{eq:aux-max-vs-max} comes from~$\cT$ and is compatible with the units there.
\end{proof}

\begin{Rem}
Recall from \Cref{Prop:A/B}\,\eqref{it:A/B-b} that $\Loc(\cB)=\Ker(\hEB\otimes-)$. What \Cref{Prop:EB=EC} tells us is that a strictly larger localizing \tensid\ $\Loc(\cB)\subsetneq\cC\subsetneq\cA$ will share the same pure-injective weak ring~$\EC=\EB$ and in particular cannot be equal to $\Ker(\hEC\otimes-)$. On the other hand, we saw in Lemma~\ref{Lem:Ker[-,E]} that there is another way of constructing a localizing \tensid. Let us compare them.
\end{Rem}

\begin{Prop}
\label{Prop:Ker@<C<Ker[]}%
Let $\cC\subset\cA$ be a localizing \tensid\ and~$\EC$ the associated weak ring (\Cref{Cons:E_C}). Then we have $\Ker(\hEC\otimes-)\subseteq\cC\subseteq\Ker([-,\hEC])$.
\end{Prop}

\begin{proof}
The first inclusion can be found in~\cite{BalmerKrauseStevenson19}: Every $M\in\cA$ admits an injective hull, say $M\into F$, with $F$ necessarily $\otimes$-flat (\Cref{Rec:tensor-on-MT}). Suppose that $\hEC\otimes M=0$. Down in~$\AC$, tensoring $\bar M\into \bar F$ with $\bar\eta_\cC\colon \bunit\into \bEC$ we see that $\bEC\otimes \bar M=0$ forces $\bar M=0$. For the second inclusion, we have by \Cref{Prop:A/C}\,\eqref{it:A/C-e}
\[
[-,\hEC]_{\cA}\cong[-,R_{\cC}(\bEC)]_{\cA}\cong R_{\cC}\big([Q_{\cC}(-),\bEC]_{\AC}\big).
\]
It follows that $\cC$, which is $\Ker(Q_{\cC})$, is contained in $\Ker([-,\hEC]_{\cA})$ as claimed.
\end{proof}

We have made all the preparation for the following result.
\begin{Thm}
\label{Thm:max-to-max}%
Under~\Cref{Hyp:big-tt}, let $\cB\in\SpchT$ be a maximal Serre \tensid\ of~$\cA\fp$. Let $\EB$ be its associated weak ring in~$\cT$ (\Cref{Cons:E_C}) and
\begin{equation}
\label{eq:max-vs-max}%
{\cB'}=\Ker([-,\hEB])=\SET{M\in\MT}{[M,\hEB]=0}
\end{equation}
where $[-,-]$ denotes the internal-hom functor in~$\cA=\MT$ (see Lemma~\ref{Lem:Ker[-,E]}). Then $\cB'\subsetneq\cA$ is a maximal localizing \tensid\ such that $\cB'\cap \cA\fp=\cB$. Furthermore, $\cB'$ is the \emph{unique} maximal localizing \tensid\ containing~$\cB$.
\end{Thm}

\begin{proof}
We have seen in Lemma~\ref{Lem:Ker[-,E]} that such $\cB'$ is indeed a localizing \tensid. As $\hat\eta_\cB\colon \hunit\to \hEB$ is non-zero, $\cB'$ is certainly proper. It contains~$\cB$ by \Cref{Prop:Ker@<C<Ker[]}.

To show both that $\cB'$ is maximal and unique, it suffices to show that if $\cC\subsetneq\cA$ is a localizing \tensid\ that contains~$\cB$ then $\cC\subseteq\cB'$. To see this, note that in that case $\EB\simeq\EC$ by \Cref{Prop:EB=EC}. Using \Cref{Prop:Ker@<C<Ker[]} again, we conclude that $\cC\subseteq\Ker([-,\hEC])=\Ker([-,\hEB])=\cB'$ as claimed.
\end{proof}

We can easily compare the respective residue functors of~$\cB$ and~$\cB'$.

\begin{Prop}
\label{Prop:hB-hBp}%
Let $\cB\in\SpchT$ be a homological prime and ${\cB'}=\Ker([-,\hEB])$ the maximal localizing \tensid\ of \Cref{Thm:max-to-max}. We have a commutative diagram
\begin{equation}
\label{eq:hB-hBp}%
\vcenter{\xymatrix@C=3em{
\cT \ar[r]^-{\yoneda} \ar@/^2em/[rr]^-{\hB} \ar@/_1em/[rrd]_-{\hBp}
& \cA \ar@{->>}[r]^-{Q_\cB} \ar@{->>}[rd]_(.35){Q_{\cB'}\!}
& \AB \ar@{->>}[d]^-{\bar Q}
\\
&& \ABp
}}
\end{equation}
where $\bar Q\colon \AB\onto \ABp$ is the quotient of~$\AB$ by its unique maximal localizing \tensid. In particular, the only localizing \tensids\ of~$\ABp$ are zero and~$\ABp$.
\end{Prop}

\begin{proof}
By \Cref{Thm:max-to-max}, the maximal localizing \tensid\ ${\cB'}$ is the unique one containing~$\cB$. The Gabriel quotient~$\ABp$ of~$\cA$ by~${\cB'}$ can then be done in two steps, first by~$\Loc(\cB)$, giving us~$\AB$, and then by the maximal localizing \tensid\ of~$\AB$.
\end{proof}

\begin{Rem}
We must confess that we do not know of an example where $\cB'\neq\Loc(\cB)$, \ie where the above $\bar Q$ is not an equivalence. So it is possible that the construction of \Cref{Thm:max-to-max} simply sends $\cB\in\SpcT$ to $\Loc(\cB)$. This would come as a surprise to the author though.
\end{Rem}

\begin{Rem}
Here are some possible misconceptions:
\begin{enumerate}[\rm(a)]
\item
\label{it:misc-a}%
Perhaps when $\unit$ generates~$\cT^c$ as a triangulated category, we can forget the tensor and all localizing subcategories in~$\cA=\MT$ are automatically \tensid?
\item
Perhaps the \emph{maximal} localizing subcategories are automatically \tensid, \ie perhaps maximal localizing \tensids\ are also maximal localizing?
\item
Perhaps in those residue fields $\AB$ or~$\ABp$ every object is a sum of spheres (suspensions of the unit, or invertible objects)?
\end{enumerate}
There are implications between those claims. Unfortunately, they are all false, as the following example will show.
\end{Rem}

\begin{Exa}
Let $p\ge 5$ be a prime number, $C_p$ the cyclic group of order~$p$ and $k$ a field of characteristic~$p$. Let $\cT=\Stab(kC_p)$ be the stable module category of $kC_p$-modules modulo projectives. This is an `exotic' tt-\emph{field} of~\cite{BalmerKrauseStevenson19}. The $\otimes$-unit $\unit=k$ generates~$\cT^c$ and all big objects are coproducts of compacts, although not only of $\otimes$-invertibles (for $p\neq 2,3$). This category~$\cT$ is pure semi-simple, meaning for instance that $\yoneda\colon \cT\hook\cA$ is fully faithful, or that all objects are pure-injective. The only possible~$\EB$ or $\EBp$ is~$\unit$. This shows that the only proper Serre \tensid\ of~$\cA\fp$ and the only proper localizing \tensid\ of~$\cA$ are zero. Pretty fieldy... Yet, under $kC_p\cong k[t]/t^p$, if we write $\langle i\rangle$ for the indecomposable object $k[t]/t^i$, then one can show that the map $t\cdot\colon\langle2\rangle\to \langle2\rangle$ is seen as zero by $\unit=\langle1\rangle$ and $\Sigma(\unit)=\langle p-1\rangle$. In other words, the homological functor $H=\Homcat{\cT}^\sbull(\langle1\rangle,-)$ is not faithful. In particular, its kernel defines a non-zero proper Serre subcategory of~$\cA\fp$, that cannot be tensor-ideal.
\end{Exa}

\begin{Rem}
Let us insist a little more on misconception~\eqref{it:misc-a} above, for it might have emerged in the reader's brain during the proof of Lemma~\ref{Lem:Ker[-,E]}, where we first prove that $\Ker([-,E])$ is a localizing (hence Serre) subcategory and use this to deduce that it is a \tensid\ from just verifying that it is closed under tensoring with rigid objects. However, we did \emph{not} conclude this by simply saying that rigid objects generate~$\cA$, which is true but not sufficient by the above example.
\end{Rem}

\goodbreak
\section{Support for big objects}
\label{sec:support}%
\medbreak

We come to the central definition of this paper. See \Cref{Rec:Spch} for $\SpchT$.

\begin{Def}
\label{Def:Supp}%
The \emph{(homological) support} of an arbitrary object~$X$ in the big tt-category~$\cat{T}$ is the following subset of the homological spectrum~$\SpchT$:
\[
\Supp(X):=\SET{\cB\in\SpchT}{[X,\EB]_{\cT}\neq 0}\,.
\]
\end{Def}

\begin{Prop}
\label{Prop:Supp}%
Under the construction $\cB\mapsto \cB'$ of \Cref{Thm:max-to-max}, the above support $\Supp(X)$ of an object~$X\in\cT$ is the following
\[
\Supp(X):=\SET{\cB\in\SpchT}{\hBp(X)\neq 0\textrm{ in }\ABp}.
\]
\end{Prop}

\begin{proof}
Let $\cB\in\SpchT$ and $\cB'=\Ker([-,\hEB])$ the corresponding maximal localizing \tensid\ of~$\cA$. By~\eqref{eq:[-,E]} and conservativity of restricted-Yoneda $\yoneda\colon \cT\to \cA$, the vanishing of $[X,\EB]_\cT$ is equivalent to the vanishing of $\yoneda([X,\EB]_\cT)\cong [\hat X,\hEB]$ and the latter is equivalent to~$\hat X\in{\cB'}$ by~\eqref{eq:max-vs-max}, which in turn is equivalent to the vanishing of~$Q_{\cB'}(\hat X)=\hBp(X)$.
\end{proof}

We are now ready to prove the basic properties listed in~\Cref{Thm:support-intro}.

\begin{Prop}
\label{Prop:support}%
The support of \Cref{Def:Supp} satisfies the following properties.
\begin{enumerate}[\rm(a)]
\item
\label{it:Supp-sigma}%
We have $\Supp(\Sigma X)=\Supp(X)$ for all~$X\in\cat{T}$.
\smallbreak
\item
\label{it:Supp-coprod}%
We have $\Supp(\sqcup_{i\in I} X_i)=\cup_{i\in I}\Supp(X_i)$ for every family~$\{X_i\}_{i\in I}$ in~$\cat{T}$.
\smallbreak
\item
\label{it:Supp-triangle}%
For every distinguished triangle~$X\to Y\to Z\to \Sigma X$ in~$\cat{T}$, we have $\Supp(Z)\subseteq\Supp(X)\cup\Supp(Y)$.
\end{enumerate}
\end{Prop}

\begin{proof}
These are immediate consequences of the description of $\Supp(X)$ in \Cref{Prop:Supp} and the fact that $\hBp\colon \cT\to \ABp$ is coproduct-preserving, homological and compatible with suspension (see \Cref{Rem:Sigma}) for every $\cB\in\SpchT$.
\end{proof}

Let us check agreement with the usual support on compact-rigid objects.

\begin{Prop}
\label{Prop:agree}%
For every $x\in\cT^c$, the three notions of support coincide:
\[
\Supp(x)=\SET{\cB\in\SpchT}{\hB(x)\neq 0}=\phi\inv(\supp(x))
\]
where $\phi\colon \SpchT\onto \SpcT$ is as in \Cref{Rec:Spch}. In particular, $\Supp(0)=\varnothing$ and $\Supp(\unit)=\SpchT$. Consequently, if $\Supp(x)=\varnothing$ for~$x\in\cT^c$ then $x=0$.
\end{Prop}

\begin{proof}
Restricted-Yoneda on~$\cT^c$ is the actual Yoneda, $\cT^c\hook\mT$, and in particular $\hat x\in\cA\fp$ is finitely presented for all~$x\in\cT^c$. Hence for $\cB\in\SpchT$ and $x\in\cT^c$, the condition $[\hat x,\hEB]=0$, meaning $\hat x\in\cB'$, is equivalent to $\hat x\in\cB'\cap \cA\fp=\cB$. The latter reads $x\in\yoneda\inv(\cB)=\phi(\cB)$. This gives the main part. Since $\phi$ is surjective, $\Supp(x)=\varnothing$ implies $\supp(x)=\varnothing$ which in turn forces $x=0$ as $\cT^c$ is rigid.
\end{proof}

Let us check the Tensor-Product Formula. With \Cref{Thm:max-to-max} and \Cref{Prop:Supp} under our belt, it is now easy.
\begin{Thm}
\label{Thm:TPF}%
We have $\Supp(X\otimes Y)=\Supp(X)\cap\Supp(Y)$ for all $X,Y\in \cat{T}$.
\end{Thm}

\begin{proof}
Using the description of the support given in \Cref{Prop:Supp}, and using that $\hBp\colon \cT\to \ABp$ is a tensor functor, it suffices to show that $\hBp(X)\otimes \hBp(Y)=0$ implies $\hBp(X)=0$ or $\hBp(Y)=0$, the other direction being obvious. Suppose that $\hBp(X)\neq 0$. Then since $\hBp(X)$ is $\otimes$-flat, the subcategory $\Ker(\hBp(X)\otimes-)$ is a proper localizing \tensid\ of~$\ABp$, hence it must be zero (see \Cref{Prop:hB-hBp}). So $\hBp(X)\otimes \hBp(Y)=0$ forces $\hBp(Y)=0$.
\end{proof}

\begin{Rem}
A heads-on proof of \Cref{Thm:TPF} from \Cref{Def:Supp}, \ie from the vanishing of $[X,\EB]$ and $[Y,\EB]$, would not be that easy. The difficulty has been entirely pushed in the translation $\cB\leadsto\cB'$ between maximal Serre \tensids\ of~$\cA\fp$ and maximal localizing \tensids\ of~$\cA$, obtained in~\Cref{Thm:max-to-max}.
\end{Rem}

Another case where the naive support~\eqref{eq:naive-support} works is the following:

\begin{Thm}
\label{Thm:rings}%
Under~\Cref{Hyp:big-tt}, let $A$ be a weak ring in~$\cT$ (\Cref{Def:weak-ring}). Then we have
\[
\Supp(A)=\SET{\cB\in\SpchT}{\hB(A)\neq 0}.
\]
Furthermore, if $\Supp(A)=\varnothing$ then $A=0$.
\end{Thm}

\begin{proof}
We have done some preparation. Again, let $\cB\in\SpchT$ and ${\cB'}=\Ker([-,\EB])$ the corresponding maximal localizing \tensid. In view of \Cref{Prop:hB-hBp}, we only need to show that if $\hB(A)\neq0$ then $\hBp(A)\neq0$. In \Cref{Prop:ring-in-AB}, we saw that if $\hB(A)\neq0$ then its unit $\bunit\into \hB(A)$ is a monomorphism. This uses that $\hB(A)$ is $\otimes$-flat in~$\AB$ by \Cref{Prop:A/C}\,\eqref{it:A/C-b}. Applying the exact functor $\bar Q\colon \AB\onto \ABp$ coming from \Cref{Prop:hB-hBp}, we have a monomorphism $\bunit\into \hBp(A)$ in~$\ABp$, forcing~$\hBp(A)\neq0$. For the `furthermore' part, suppose that $\Supp(A)=\varnothing$, \ie that $\hB(A)=0$ for all~$\cB\in\SpchT$. Then we can invoke the Nilpotence Theorem of~\cite{Balmer20a} for the unit $\eta\colon\unit\to A$ of the weak ring~$A$. Clearly, $\hB(\eta)=0$ for all~$\cB$, since the target of these maps are all zero. Hence $\eta\potimes{n}=0$ for some $n\gg1$. However, $A\otimes\eta\colon A\into A\otimes A$ is a split monomorphism by definition of a weak ring, hence $A$ is a direct summand of~$A\potimes{n}$ for all~$n\ge1$, by induction. Combined with $A\potimes{n}=0$ we get $A=0$ as claimed.
\end{proof}

\begin{Exa}
\label{Exa:Supp(EB)}%
Let $\cB\in\SpchT$. Then $\Supp(\EB)=\{\cB\}$. Indeed, by \Cref{Thm:rings}, we can use the `naive' support $\Supp(\EB)=\SET{\cC\in\SpchT}{\EB\otimes\EC\neq 0}$. But we know that $\EB\otimes\EB\neq 0$ (as $\EB$ is a direct summand) whereas $\EB\otimes\EC=0$ for all~$\cB\neq\cC$; see~\cite[Prop.\,5.3]{Balmer20a}.
\end{Exa}

\begin{Cor}
\label{Cor:[EB,EC]}%
For $\cB$ and $\cC$ distinct homological primes in~$\SpchT$, we have $[\EB,\EC]=0$ in~$\cT$ and~$[\hEB,\hEC]=0$ in~$\cA$.
\qed
\end{Cor}

\goodbreak
\section{The image of $\textrm{Spc}(F)$}
\label{sec:im(Spc(F))}%
\medbreak

%
\begin{Rec}
\label{Rec:T-S}%
We consider a tensor-triangulated functor $F\colon \cT\to \cS$ between big tt-categories as in \Cref{Hyp:big-tt}. Because every tensor functor preserves rigidity and because we assume that rigid and compact objects coincide, $F$ restricts to a tt-functor $F\colon \cT^c\to \cS^c$. We assume that our functor~$F$ preserves coproducts. By Brown-Neeman Representability~\cite{Neeman96}, this is equivalent to the existence of a right adjoint~$U\colon \cS\to \cT$, that is then lax-monoidal since~$F$ is monoidal. This implies that $U(\unit)$ is a commutative ring object in~$\cT$. In fact, since $F$ preserves compacts, $U$ preserves coproducts (and thus admits another right adjoint).

By Krause~\cite[Cor.\,2.4]{Krause00}, both~$F$ and~$U$ induce exact coproduct-preserving functors $\hat F\colon \MT\to \MS$ and $\hat U\colon \MS\to \MT$ such that $\hat F\circ \yoneda_\cT\cong \yoneda_\cS\circ \hat F$ and $\hat U\circ \yoneda_\cS\cong \yoneda_\cT\circ \hat U$. These can be described by the usual Kan formulas
\[
\hat F(M)=\mathop{\colim}_{(\hat x\to M)\in (\yoneda_\cT/M)}\yoneda_\cS(F(x))
\qquadtext{and}
\hat U(N)=\mathop{\colim}_{(\hat y\to N)\in (\yoneda_\cS/N)}\yoneda_\cT(U(y)).
\]
Alternatively, when viewing~$N\in\MS$ as an additive functor $\cS^c\to \Ab$, we have $\hat U(N)=N\circ F\colon \cT^c\to \Ab$. Since $\otimes$ is colimit-preserving and~$F$ is symmetric monoidal, it follows that $\hat F$ is also symmetric monoidal. The adjunction~$F\adj U$ yields an adjunction~$\hat F\adj \hat U$ making the following diagram of adjunctions commute:
\begin{equation}
\label{eq:hatF-hatU}%
\vcenter{\xymatrix@R=1.8em{
\cT \ar@<-.3em>[d]_-{F} \ar[r]^-{\yoneda_\cT}
& \MT \ar@<-.3em>[d]_-{\hat F}
\\
\cS {}^{\vcorrect{.6}} \ar@<-.3em>[u]_-{U} \ar[r]^-{\yoneda_\cS}
& \MS.\!\! {}^{\vcorrect{.6}} \ar@<-.3em>[u]_-{\hat U}
}}
\end{equation}
Note also that $\hat F$ preserves finitely presented objects~$\hat F(\mT)\subseteq\mS$ since they are generated by the compact objects. See details in~\cite[Construction~6.10]{BalmerKrauseStevenson19}. Finally, since $F$ and~$U$ satisfy a projection formula $U(Y\otimes F(X))\cong U(Y)\otimes X$ by \cite[2.16]{BalmerDellAmbrogioSanders16}, the same holds for~$\hat F$ and~$\hat U$:
\begin{equation}
\label{eq:proj-formula}
\hat U(N\otimes \hat F(M))\cong \hat U(N)\otimes M
\end{equation}
by using that $\hat F$, $\hat U$ and~$\otimes$ commute with colimits and the above explicit formulas.
\end{Rec}

\begin{Rem}
We want to show functoriality of~$\Spch(-)$. The following result is not entirely obvious, as the commutative algebraists will recognize. For a homomorphism of commutative rings, it is not true in general that the preimage of a maximal ideal is maximal. This vindicates again the use of \emph{maximal} ideals~$\cB$ in~$\mT$ when constructing~$\SpchT$, as opposed to some kind of prime ideals.
\end{Rem}

The following is extracted from~\cite{Balmer20a}:
\begin{Prop}
\label{Prop:Finv}%
Let $\cT$ be a big tt-category as in \Cref{Hyp:big-tt}. Let $H\colon \cT\to \cD$ be a homological functor to a locally coherent Grothendieck category~$\cD$ admitting a colimit-preserving tensor. Assume that~$H$ is monoidal {\rm(\footnote{\,not necessarily \emph{symmetric} monoidal})} and maps compact objects of~$\cT$ to finitely presented objects. Assume furthermore that $H(X)$ is $\otimes$-flat in~$\cD$ for every~$X\in\cT$. Let $\hat H\colon \cA=\MT\to \cD$ be the exact coproduct-preserving functor induced by~$H$. Then $\Ker(\hat H)$ is a localizing \tensid\ generated by its finitely presented part~$\cB:=\Ker(\hat H)\cap \cA\fp$. If moreover $\cD\fp$ is simple, \ie has only zero and~$\cD\fp$ as Serre \tensids, then $\cB$ is a homological prime, \ie it is maximal in~$\cA\fp$.
\end{Prop}

\begin{proof}
This is the first page of the proof of~\cite[Thm.\,5.6]{Balmer20a}. As the notation in \loccit\ depends on some triangular primes~$\cP$, we outline a cleaned-up version here for the reader's convenience. The kernel $\Ker(\hat H\colon \cA\to \cD)=\Loc(\cB)$ is generated by its finitely presented part~$\cB$ by~\cite[Prop.\,A.6]{BalmerKrauseStevenson19}. To show that $\cB$ is maximal when $\cD\fp$ is simple, assume that $\cC\supsetneq\cB$ is larger and show $\cC=\cA$. For that, pick $M\in\cC\smallsetminus\cB$, so that $M\otimes\hEB\neq 0$ whereas $M\otimes\hEC=0$. From the former deduce that $\hat H(M)\neq 0$ in~$\cD\fp$ and from the latter deduce that $\hat H(M)\otimes H(\EC)=0$ in~$\cD$. Use simplicity of~$\cD\fp$ and $\otimes$-flatness of~$H(\EC)$ to deduce that $H(\EC)=0$, \ie $\hEC\in\Ker(\hat H)=\Loc(\cB)\subset\Loc(\cC)$, forcing $\hunit\in\cC$ as well, from the exact sequence $J\into \hunit\to \hEC$ where $J\in\Loc(\cC)$, as always, and $\EC\in\Loc(\cC)$ as just proved.
\end{proof}

\begin{Lem}
\label{Lem:Finv}%
With notation as in~\Cref{Rec:T-S}, let $\cC\in\SpchS$.
Then $\cB:=\SET{M\in\mT}{\hat F(M)\in\cC}$ is a maximal Serre \tensid\ of~$\mT$. Furthermore, we have a diagram
\begin{equation}
\label{eq:F-res-fields}%
\vcenter{\xymatrix@R=1.5em{
\cT \ar@<-.3em>[dd]_-{F} \ar[r]^-{\yoneda_\cT}
& \cA=\MT {}_{\vcorrect{.6}} \ar@<-.3em>[dd]_-{\hat F} \ar@{->>}@<-.2em>[rd]_-{Q_{\cB}}
\\
&& \AB=\MT/\Loc(\cB) {}^{\vcorrect{.6}} \ar@<-.3em>[dd]_-{\bar F} \ar@{_(->}@<-.2em>[lu]_-{R_{\cB}}
\\
\cS \ar@<-.3em>[uu]_-{U} \ar[r]^-{\yoneda_\cS}
& \MS {}_{\vcorrect{.6}} \ar@<-.3em>[uu]_-{\hat U} \ar@{->>}@<-.2em>[rd]_-{Q_{\cC}}
\\
&& \MS/\Loc(\cC) {}^{\vcorrect{.6}} \ar@<-.3em>[uu]_-{\bar U} \ar@{_(->}@<-.2em>[lu]_-{R_{\cC}}
}}
\end{equation}
with $\yoneda_\cS\,F\cong\hat F\,\yoneda_\cT$ and $\yoneda_\cT\,U\cong\hat U\,\yoneda_\cS$ and $Q_{\cC}\,\hat F\cong\bar F\,Q_{\cB}$ and $\hat U\,R_{\cC}\cong R_{\cB}\,\bar U$.
Finally, the pure-injective $\EB$ is a direct summand of~$U(\EC)$ in~$\cT$ (see \Cref{Cons:E_C}).
\end{Lem}

\begin{proof}
The left-hand square of~\eqref{eq:F-res-fields} is~\eqref{eq:hatF-hatU}, repeated for cognitive help. Since $\hat F$ is exact, $\hat U$ preserves injectives hence $U(\EC)$ is pure-injective. Since $U$ is lax-monoidal, $U(\EC)$ is a weak ring, with unit~$\unit\to U(\EC)$ adjoint to $\eta_{\cC}\colon \unit_\cS=F(\unit_\cT)\to \EC$. In particular, $U(\EC)\neq 0$. The internal-hom version of the adjunction reads
\begin{equation}
\label{eq:aux-internal-adj}%
[M,\hat U(\hEC)]\cong \hat U[\hat F(M),\hEC]
\end{equation}
for every~$M\in\cA$, as can be checked by testing under $\Homcat{A}(\hat x,-)$ for $x\in\cT^c$.

Consider now the functor $H:=Q_{\cC}\circ \hat F\circ\yoneda_\cT$
\[
H\colon \cT\to \MT\to \MS\onto\MS/\Loc(\cC)=:\cD.
\]
This is a homological functor to which we can apply \Cref{Prop:Finv}, with $\hat H$ being necessarily~$Q_{\cC}\circ\hat F$. It tells us that $\cB=\Ker(\hat H)\cap\cA\fp$ is indeed a maximal Serre \tensid\ of~$\cA\fp$. We can then factor~$\hat H$ via~$Q_\cB$, yielding a unique tensor-exact coproduct-preserving functor $\bar F\colon \MT/\Loc(\cB)\to \MS/\Loc(\cC)$ making the following diagram commute
\[
\xymatrix{
\MT \ar[d]_-{\hat F} \ar@{->>}[r]^-{Q_{\cB}} \ar[rd]_-{\hat H}
& \MT/\Loc(\cB)=\AB \ar@{..>}[d]^-{\bar F}_-{\exists\,!}
\\
\MS \ar@{->>}[r]_-{Q_{\cC}}
& \MS/\Loc(\cC) \ar@{..>}@/_3em/[u]_-{\bar U}
}
\]
By general Grothendieck-category theory, $\bar F$ admits a right adjoint~$\bar U$. This gives us~\eqref{eq:F-res-fields}. Note that $R_{\cB}\bar U$ is the right adjoint to~$\bar F\,Q_{\cB}\cong Q_{\cC}\,\hat F$, whose right adjoint is also~$\hat U\,R_{\cC}$. Hence we have
\begin{equation}
\label{eq:aux-RU}%
R_{\cB}\circ \bar U \cong \hat U\circ R_{\cC}\,.
\end{equation}
Consider $\bunit\into \bEC$ the injective hull of the unit in~$\MS/\Loc(\cC)$, as in the statement. The pure-injective $\EC\in\cS$ corresponding to~$\cC$ is then characterized by~$\hEC=R_{\cC}(\bEC)$. Therefore, by~\eqref{eq:aux-RU}, we have $R_{\cB}(\bar U(\bEC))\cong \hat U(\hEC)=\yoneda_{\cT}(U(\EC))$. Applying~$Q_{\cB}$ to this relation, we see that the injective $\bar U(\bEC)$ is the image in~$\AB$ of~$U(\EC)$. By \Cref{Prop:ring-in-AB} the unit of this non-trivial weak ring $\bar U(\bEC)$ is a monomorphism $\bunit\into \bar U(\bEC)$ in~$\AB$ and $\bar U(\bEC)$ is injective. Therefore the injective hull $\bEB$ of~$\bunit$ is a direct summand of~$\bar U(\bEC)$. As usual, this holds in~$\cT$ already, by \Cref{Prop:A/C}\,\eqref{it:A/C-d}.
\end{proof}

Summarizing our discussion:
\begin{Thm}
\label{Thm:Spc(F)}%
Let $F\colon \cT\to \cS$ be a coproduct-preserving tt-functor between big tt-categories. Then the map $\Spch(F)\colon\SpchS \to \SpchT$
\[\begin{array}{rcl}
\SpchS & \to & \SpchT
\\
\cC & \mapsto & \hat F\inv(\cC)
\end{array}\]
is well-defined and makes the following square commute
\begin{equation}
\label{eq:Spc(F)}%
\vcenter{\xymatrix@C=5em{
\SpchS \ar[r]^-{\Spch(F)} \ar@{->>}[d]_-{\phi_{\cS}}
& \SpchT \ar@{->>}[d]^-{\phi_{\cT}}
\\
\SpcS \ar[r]^-{\Spc(F)}
& \SpcT
}}
\end{equation}
where~$\phi\colon \Spch\to \Spc$ is as in \Cref{Rec:Spch}.
\end{Thm}

\begin{proof}
We use \Cref{Lem:Finv} to show that the map~$\varphi$ is well-defined. The commutativity of the square then follows from $\hat F \circ \yoneda_{\cT}\cong\yoneda_{\cS}\circ F\colon \cT^c\to \MS$.
\end{proof}

\begin{Thm}
\label{Thm:Spc-image}%
Let $F\colon \cT\to \cS$ be a tt-functor between big tt-categories as in \Cref{Hyp:big-tt} with a right adjoint $U\colon \cS\to \cT$ (\Cref{Rec:T-S}). Then the image of $\Spch(F)\colon \SpchS\to \SpchT$ is equal to
$\Supp(U(\unit))$.
\end{Thm}

\begin{proof}
Let $\psi=\Spch(F)\colon \SpchS\to \SpchT$ and $U(\unit)\in\cT$. Since $U(\unit)$ is a ring object, we have $\Supp(U(\unit))=\SET{\cB\in\SpchT}{\hB(U(\unit))\neq 0}$ by \Cref{Thm:rings}.

The inclusion $\Img(\psi)\subseteq\Supp(U(\unit))$ is relatively easy. Suppose that $\cB$ belongs to the image of~$\psi$ and let $\cC\in\SpchS$ such that $\cB=\hat F\inv(\cC)$. We can then apply \Cref{Lem:Finv} and we have in particular the diagram~\eqref{eq:F-res-fields}. The lax-monoidal functor~$\hat U$ maps~$\hunit_\cS$ to a commutative ring object~$\hat U(\hat\unit)=\yoneda_{\cT}(U(\unit))$ which acts on every object of the form~$\hat U(Y)$, since $\hunit_\cS$ acts on every object~$Y\in\MS$. (This action is as a `module' but we avoid this terminology since we already mean something else by `modules'.) We apply this to $Y=\hEC$ the injective object corresponding to~$\cC$. So, the ring $\hat U(\unit)$ acts on $\hat U(\hEC)$ in~$\cA=\MT$. Applying the tensor functor~$Q_\cB\colon \cA\onto \AB$, we see that the commutative ring object~$\hB(U(\unit))$ acts on $\hB(\hEC)$ in~$\AB$. In particular if, \ababs, the ring $\hB(U(\unit))$ vanished then so would the object $\hB(U(\hEC))$. By \Cref{Lem:Finv}, we also know that $\EB$ is a direct summand of~$U(\hEC)$, hence $\hB(\EB)$ is a direct summand of~$\hB(U(\hEC))=0$. This implies the vanishing the injective hull $\hB(\EB)=\bEB$ of~$\bunit$ in~$\AB$, a contradiction. So $\hB(U(\unit))$ cannot be zero, as claimed.

Conversely, let $\cB$ be such that $\hB(U(\unit))\neq 0$ and let us show that $\cB\in\Img(\psi)$. Let $I$ be the kernel of $\hat\eta_\cB\colon \hunit\to\hEB$ in~$\cA$. We have an exact sequence $I\into \hunit \to \hEB$ and therefore $\hat F(I)\into \hat F(\hunit)\to \hat F(\hEB)$ in~$\MS$. Note that $\hat F(\hEB)=\widehat{F(\EB)}$ is flat. Consider in~$\mS$ the Serre \tensid
\[
\cC_0:=\Ker\big((\hat F(\hEB)\otimes-)\restr{\mS}\big)=\SET{N\in\mS}{\hat F(\hEB)\otimes N=0}.
\]
If, \ababs, $\cC_0=\mS$ then $\hat F(\hEB)=0$, hence $\hat F(I)=\hunit$. Since $I=\colim_{\textrm{f.p.\,}M\into I}M$ as in \Cref{Prop:A/B}\,\eqref{it:A/B-c}, and since $\hunit_{\cS}$ is finitely presented, there exists $M\into I$ finitely presented such that $\hat F(M)=\hunit$ already. Note that $M\into I\in\Loc(\cB)$ implies $M\in\cB$.  Now, we use $\hat U\hat F\cong \hat U(\unit)\otimes-$ by the projection formula~\eqref{eq:proj-formula}. In short we have $M\otimes \hat U(\unit)\cong \hat U(\unit)$ with $M\in \cB\subset\Ker(\hEB\otimes-)$. This implies that $\hEB\otimes \hat U(\unit)\cong \hEB\otimes M\otimes \hat U(\unit)\cong 0$ and thus $\EB\otimes U(\unit)=0$ meaning that $\cB\notin\Supp(U(\unit))$, a contradiction. So $\cC_0$ is a proper \tensid\ of~$\mS$. Then choose any maximal \tensid\ $\cC\in\SpchS$ containing~$\cC_0$. Then $\hat F\inv(\cC)\supseteq\cB$ by construction hence by maximality $\cB=\hat F\inv(\cC)=\psi(\cC)$ and indeed $\cB\in\Img(\psi)$.
\end{proof}

\begin{Cor}
With hypotheses as in \Cref{Thm:Spc-image}, the image of the ordinary map $\Spc(F)\colon\SpcS\to \SpcT$ on triangular spectra is~$\phi(\Supp(U(\unit)))$, where $\phi\colon \Spch\to\Spc$ is as in \Cref{Rec:Spch}.
\end{Cor}

\begin{proof}
Recall the commutative diagram~\eqref{eq:Spc(F)}. Since $\phi_{\cS}$ is onto, the image of~$\Spc(F)$ is also the image of $\phi_{\cT}\circ\Spch(F)$, that is, the image under~$\phi_\cT$ of the image of~$\Spch(F)$ determined in \Cref{Thm:Spc-image} to be $\Supp(U(\unit))$.
\end{proof}

\begin{Rem}
We emphasize that $\phi\colon \SpchT\onto \SpcT$ is a bijection in many examples by \cite[\S\,5]{Balmer20a}. In practice, it is very common that a tt-functor on the small part $F\colon \cT^c\to \cS^c$ is the restriction of coproduct-preserving tt-functor $F\colon \cT\to \cS$. So the last corollary provides a vast improvement on~\cite{Balmer18a}.
\end{Rem}

\begin{appendix}
\goodbreak
\section{Comparing homological and triangular spectra}
\label{App:hom-local}%
\medbreak

We want to discuss the injectivity of the map $\phi\colon \SpchT\onto \SpcT$, from the homological spectrum (\Cref{Rec:Spch}) to the usual triangular spectrum~$\SpcT$. By a general finite-smashing localization argument, we can reduce to the case where $(0)\in \SpcT$, meaning that $\cT^c$ is \emph{local}. For instance, $\cT$ could be $\SH_{(p)}$ the $p$-local stable homotopy category or $\Der(R)$ the derived category of a commutative local ring~$R$. In that case, we have the following characterization:

\begin{Thm}
\label{Thm:hom-local}%
Let $\cT$ be a big tt-category as in \Cref{Hyp:big-tt} and assume its subcategory of rigid-compact~$\cT^c$ is \emph{local}, meaning that $c\otimes d=0$ with $c,d\in\cT^c$ forces $c=0$ or $d=0$. Then the following are equivalent:
\begin{enumerate}[\rm(i)]
\item
\label{it:fiber=*}%
The fiber of $\phi\colon \SpchT\onto \SpcT$ over~$(0)$ is reduced to a single point.
\smallbreak
\item
\label{it:hom-local}%
Whenever two maps~$f$, $g$ in~$\cT^c$ satisfy $f\otimes g=0$ then either $f$ is \tensnil\ on some non-zero compact (that is, there exists $z\in\cT^c$, $z\neq0$ such that $f\potimes{n}\otimes z=0$ for $n\gg1$) or $g$ is \tensnil\ on some non-zero compact.
\end{enumerate}
\end{Thm}

\begin{proof}
Suppose~\eqref{it:fiber=*} and let us prove~\eqref{it:hom-local}. By standard adjunction tricks using rigidity, we can assume that $f\colon \unit\to x$ and $g\colon \unit\to y$ both start from~$\unit$. Let $\cA=\MT$ as before and consider the subcategory of~$\cA\fp=\mT$ given as follows
\[
\cB_0:=\SET{M\in\cA\fp}{\textrm{ there exists }z\in\cT^c,\,z\neq 0\textrm{ with }\hat z\otimes M=0}.
\]
Since $\cT^c$ is local, $\cB_0$ is closed under extensions, hence $\cB_0$ is a Serre \tensid\ of~$\cA\fp$. Consider the Gabriel quotient
\[
Q\colon \cA\fp\onto \tilde\cA\fp:=\cA\fp/\cB_0.
\]
Let us write $\tilde z\in\tilde\cA\fp$ for $Q(\hat z)$ for all~$z\in\cT^c$. The maximal \tensids\ of~$\tilde\cA\fp$ are in one-to-one correspondence with those of~$\cA\fp$ containing~$\cB_0$. We claim that these are exactly those in~$\phi\inv(0)$. Indeed, the kernel of the unit $\coev\colon \unit\to \hat z^\vee\otimes \hat z$ in~$\cA\fp$, being killed by~$\hat z$ (in~$\cT^c$ already), becomes zero in~$\tilde\cA\fp$ whenever~$z\neq 0$. In other words, $\tilde\unit$ is a subobject of $\tilde z^\vee\otimes\tilde z$ for every $z\in\cT^c$ non-zero and in particular no proper \tensid\ of~$\tilde\cA\fp$ contains any~$\tilde z$ for $z\neq 0$. Hence the preimage $\phi(\cB)=\yoneda\inv(\cB)=(0)$ is zero for every maximal \tensid\ of~$\cA\fp$ containing~$\cB_0$. Under~\eqref{it:fiber=*}, we just proved that $\tilde\cA\fp$ has a unique maximal Serre \tensid, say~$\cB$. Consider now the kernels $M\into \tilde\unit\xto{\tilde f}\tilde x$ and $N\into \tilde\unit\xto{\tilde g}\tilde y$ in~$\tilde\cA\fp$. Since $\tilde f\otimes M=0$ and $\tilde g\otimes N=0$ by the standard argument (using flatness of~$\tilde x$ and $\tilde y$), we see that $\tilde f$ is \tensnil\ on the Serre \tensid\ $\ideal{M}$ generated by~$M$ and similarly for~$\tilde g$ on~$\ideal{N}$. If, \ababs, both $\ideal{M}$ and $\ideal{N}$ are proper then they are both contained in the unique maximal Serre \tensid~$\cB$ of~$\tilde\cA\fp$. Descending to~$\tilde\cA\fp/\cB$, we therefore have two monomorphisms $\bar f\colon \bunit\into \bar x$ and $\bar g\colon \bunit\into \bar y$, with $\bar x$ and $\bar y$ $\otimes$-flat. This implies that $\bar f\otimes \bar g\colon \bunit\into \bar x\otimes\bar y$ is still a monomorphism. This contradicts the assumption that $f\otimes g=0$. Consequently, one of the Serre \tensids\ $\ideal{M}$ or~$\ideal{N}$ contains~$\tilde\unit$, hence $\tilde f$ or $\tilde g$ is \tensnil\ in $\tilde\cA\fp=\cA\fp/\cB_0$, and we conclude by definition of~$\cB_0$ that $f$ or $g$ is nilpotent on a non-zero compact.

Conversely, suppose~\eqref{it:hom-local} and let us prove~\eqref{it:fiber=*}. Assume \ababs, that there exist two different homological primes $\cB,\cC\in\phi\inv((0))$ in the fiber of~$\phi$ above zero. Then by \cite[Prop.\,5.3]{Balmer20a}, we have $\EB\otimes \EC=0$ and in particular the tensor of $\eta_\cB\colon \unit\to \EB$ with $\eta_\cC\colon \unit \to \EC$ is zero. It follows from the fact that $\hEB$ and $\hEC$ are colimits of representables and from finite presentation of~$\hunit$ that there exist factorizations $\eta_\cB\colon \unit\xto{f} x\to \EB$ and $\eta_\cC\colon \unit\xto{g} y\to \EC$ with $f\otimes g=0$ and $x,y\in\cT^c$. By~\eqref{it:hom-local}, one of~$f$ or~$g$ is \tensnil\ on some non-zero compact, say~$f$ is. Hence there exists $z\neq0$ in~$\cT^c$ such that $f\potimes{n}\otimes z=0$ for some $n\gg1$. But then $\eta_\cB\potimes{n}\otimes z=0$. Since $\EB$ is a weak ring, $\eta_\cB\otimes\EB$ is a split monomorphism, hence $\eta_\cB\potimes{n}\otimes z=0$ forces $\EB\otimes z=0$. The latter implies $0\neq z\in \cB$, or $\phi(\cB)=\yoneda\inv(\cB)\neq 0$, contradicting the choice of~$\cB$ in the fiber of~$(0)$. So the existence of two points in that fiber is absurd.
\end{proof}

\begin{Exa}
\label{exa:local}%
There are of course many examples where \Cref{Thm:hom-local} applies, since we know that $\phi$ is a bijection in all examples where $\SpchT$ and~$\SpcT$ have been computed, see~\cite[\S\,5]{Balmer20a}. One can also cook up a direct verification of~\eqref{it:hom-local} in \Cref{Thm:hom-local} in the derived category of a local ring, for instance. It is conceivable that my stable $\infty$-friends could find a direct general proof of~\eqref{it:hom-local} in \Cref{Thm:hom-local}, without assuming knowledge of~$\SpcT$, thus giving an abstract proof that $\phi$ is a bijection in some generality. The challenge is open!
\end{Exa}

\end{appendix}


\end{document}